\documentclass[11pt,letterpaper]{amsart}
\usepackage[utf8]{inputenc}
\usepackage[T1]{fontenc}
\usepackage{amsmath,amssymb,graphicx,url}
\usepackage{xcolor}
\usepackage{hyperref}
\usepackage[margin=3cm]{geometry}
\usepackage{xargs}
\input xy
\xyoption{all}

\newtheorem{prop}{Proposition}[section]
\newtheorem{coro}[prop]{Corollary}
\newtheorem{thm}[prop]{Theorem}
\newtheorem{lemma}[prop]{Lemma}
\newtheorem{definition}[prop]{Definition}
\newtheorem{construction}[prop]{Construction}
\newtheorem{example}[prop]{Example}

\DeclareMathOperator{\Aut}{Aut}

\newcommand{\GG}{\mathcal{G}}

\newcommand{\KK}{\mathcal{K}}
\newcommand{\CC}{\mathcal{C}}
\newcommand{\TT}{\mathcal{T}}
\newcommand{\EXT}{\operatorname{Ext}}

\usepackage[loadshadowlibrary,shadow, colorinlistoftodos,textsize=tiny,obeyFinal]{todonotes}

\setuptodonotes{fancyline, backgroundcolor=gray!10,bordercolor=gray}

\newcommandx{\inline}[2][1=]{\todo[inline, #1]{#2}}
\newcounter{hw}
\newcommandx{\homework}[2][1=]{\todo[inline, caption={Homework \thehw} #1]{\stepcounter{hw} #2}}

\newcommandx{\claudioLong}[2][1=Long todo, usedefault]{\todo[inline, color=red!25,caption={\textbf{Claudio:} #1}]{\textbf{Claudio: #1}. #2}}
\newcommandx{\philippeLong}[2][1=Long todo, usedefault]{\todo[inline,color=purple!25, caption={\textbf{Philippe:} #1}]{\textbf{Philippe: #1}. #2} }

\begin{document}

\title[The $j$-diagonals twisting]{coset geometries acting on their elements: The $j$-diagonals twisting}

\author[Claudio Alexandre Piedade]{Claudio Alexandre Piedade}
\address[Claudio Alexandre Piedade]{Universit\'e Libre de Bruxelles, Brussels, Belgium}
\email{claudio.piedade@ulb.be}

\author[Philippe Tranchida]{Philippe Tranchida}
\address[Philippe Tranchida]{Universit\'e Libre de Bruxelles, Brussels, Belgium}
\email{tranchida.philippe@gmail.com}

\thanks{\textit{Declaration of interests:} none}
\date{\today}

\begin{abstract}
Let $\beta$ be a coset incidence system and fix a type $j$. We use the orbits of the group of $\beta$ on pairs of distinct $j$-elements to define a Coxeter graph. The action on the $j$-elements induces an action on the corresponding Coxeter group, so the twisting construction  for coset incidence systems can be applied. The resulting coset geometry, called the $j$-diagonals twisting, is always a regular hypertope if $\beta$ is a regular hypertope. Using this, we show that finite regular hypertope whose diagram is a tree with all but one of the labels equal to four always exist. We also define an extension operation that combines this Coxeter graph with another given Coxeter graph. For regular polytopes, these constructions recover the twisting extensions of McMullen and Schulte. 
\end{abstract}

\maketitle

\section{Introduction}

Coset incidence systems are incidence systems built from a group $G$ and a collection of its subgroups $(G_i)_{i\in I}$~\cite{Tits1957,Tits1963geometries,buekenhout2013diagram}. 
Many combinatorial and geometrical objects can be realised in this form, such as abstract polytopes, vector and projective spaces, and Tits' buildings~\cite{ARP,Handbook,Tits1974}.

Given an operation on groups, there have been fruitful recent developments on how to extend these group operations to 
coset incidence system operations. For example, free products (with amalgamation) and HNN-extensions of coset incidence systems were defined in~\cite{PiedadeTranchida_FromGroupOps2026}.
The twisting operation, a type of semi-direct product, was defined in~\cite{PiedadeTranchida_FromGroupOps2026} and extended for a more general definition in~\cite{wreath}, which lead to a proper definition of a wreath product of coset incidence systems~\cite{wreath}.

In this article, we focus once more on the twisting operation. In order to use the twisting, the following data is necessary: two coset incidence systems, $\alpha$ and $\beta$, and an action $\eta$ of $\beta$ on $\alpha$ that satisfy some admissibility conditions, such as permute the parabolic subgroups and having an orbit intersection property.
We study two situations of the twisting where the action is natural and always satisfies the admissibility conditions. More precisely, given a coset incidence system $\beta = (B,(B_i)_{i\in I_\beta})$, we use the natural action of $B$ on the $j$-elements of $\beta$, for some type $j$, as a basis for the twisting. To do so, we construct a coset incidence system $\alpha= (A,(A_i)_{i\in I_\alpha})$ for a Coxeter group $A$ in such a way that its diagram has the same symmetries as the $j$-elements of $\beta$.
This construction always gives a natural action of $\beta$ on $\alpha$ which satisfies the admissibility conditions as long as $B$ acts flag-transitively on $\beta$.
Additionally, we extend this construction further for combinations of two Coxeter graphs, one of which is built from the action of $B$ on the $j$-elements of $\beta$.
Whenever $\beta$ is a regular hypertope, both of the above construction always result in a regular hypertope, whose diagrams we can control, and are described in Corollaries~\ref{coro:diagramtwistingGraph} and~\ref{coro:extensiontwistingGraph}.

These constructions, together with our previous works~\cite{PiedadeTranchida_FromGroupOps2026,wreath,halving}, allow us to construct many new regular hypertopes from known ones, and to have control on the diagrams. Moreover, the twisting of two regular hypertopes $\TT(\alpha,\beta)$ is finite if and only if both $\alpha$ and $\beta$ are finite. Hence, the tools developed here can help answering questions regarding the existence of a finite regular hypertope for a particular diagram. Indeed, one can ask the question: 
\begin{quote}
    ``Let $\mathcal{D}$ be a Coxeter diagram whose labels are all finite. Does there exists a finite regular hypertope whose diagram is $\mathcal{D}$?''
\end{quote}
Note that as soon as one edge label in equal to infinity, any regular hypertope for that diagram must be infinite as it will contain two involution whose product has infinite order.
A simple application of the twisting constructions can be used to positively answer this question whenever $\mathcal{D}$ is a tree whose edge labels are equal to $4$, except possibly for one edge label.

\subsection{Acknowledgments}
The first author is funded by an Action de Recherche Concert\'{e}e -- ARC -- from the Communaut\'{e}e Fran\c{c}aise Wallonie-Bruxelles and the second author is a Postdoctoral Researcher of the Fonds de la Recherche Scientifique -- FNRS.

\section{Background}\label{sec:prelims}

\subsection{Coset incidence systems and regular hypertopes}

An \emph{incidence system} is a quadruple $\Gamma=(X,I,*,t)$, where $X$ is a set of elements, $I$ is a set of types, $*$ is a symmetric and reflexive incidence relation, and $t:X\to I$ is a map, called the type map, such that two distinct incident elements have different types. A \emph{flag} of an incidence system is a set of pairwise incident elements, and a \emph{chamber} is a flag containing one element of every type. An incidence system is an \emph{incidence geometry} if every maximal flag is a chamber.

Let $B$ be a group and let $(B_i)_{i\in I_\beta}$ be a family of subgroups of $B$. The associated \emph{coset incidence system}
\[
\beta=(B,(B_i)_{i\in I_\beta})
\]
has the left cosets $bB_i$ as its elements of type $i$, and two cosets are defined to be incident whenever they have non-empty intersection. The group $B$ naturally acts on $\beta$ by left multiplication. For $J\subseteq I_\beta$, we write
$$B_J=\bigcap_{i\in J}B_i,
\qquad B_\emptyset=B.$$
The subgroups $B_i$ are called the \emph{maximal parabolic subgroups} of $\beta$.

Often, we are interested in incidence systems with many symmetries. The ones with full symmetries are called \textit{flag-transitive}. When the incidence system is expressed as a coset incidence system, flag-transitivity can be defined by a group theoretical condition.

\begin{thm}\cite[Proposition 2.2, adapted]{PiedadeTranchida_FromGroupOps2026}\label{thm:cosetFT}
Let $\beta = (B,(B_i)_{i\in I_\beta})$ be a coset incidence system. Then $B$ is flag-transitive on $\beta$ if and only if for every $J,H,K \subseteq I_\beta$ we have $B_H B_J \cap B_K B_J = (B_H\cap B_K ) B_J$.
Additionally, if $B$ is flag-transitive on $\beta$, then $\beta$ is a geometry.
\end{thm}

Finally, a \emph{regular hypertope} is a thin, residually connected incidence geometry with a flag-transitive automorphism group. For a more precise definition of a hypertope, and for standard facts about coset geometries and regular hypertopes, we refer to reader to~\cite{buekenhout2013diagram,PiedadeTranchida_FromGroupOps2026,hypertopes}.

\subsection{Coxeter graphs}\label{sec:prelims:Coxeter}

A \emph{Coxeter graph} $\GG$ is a simple graph whose edges are labelled by elements of $\{3,4,\ldots\}\cup\{\infty\}$. For distinct vertices $s,t\in V(\GG)$, we set $m_{s,t}=m_{t,s}$ to be the edge label of $\{s,t\}$ if $s$ and $t$ are adjacent, and we set $m_{s,t}=2$ when they are not adjacent. The Coxeter group defined by $\GG$ is the group with presentation
\[
W(\GG)=\left\langle a_s\ (s\in V(\GG))\ \middle|\
 a_s^2=1,\ (a_sa_t)^{m_{s,t}}=1\text{ whenever }m_{s,t}<\infty
\right\rangle.
\]
Thus, an edge with label $2$ is omitted, and a label $\infty$ means that no relation is imposed on the corresponding product of generators.

For $s\in V(\GG)$, set
\[
W(\GG)_s=\langle a_t\mid t\in V(\GG)\setminus\{s\}\rangle.
\]
Hence, given a graph $\GG$ and its Coxeter group $W(\GG)$, we can construct the \emph{standard coset geometry} for $W(\GG)$:
\[
\bigl(W(\GG),(W(\GG)_s)_{s\in V(\GG)}\bigr)
\]
This coset geometry is always a regular hypertope; see, for example,~\cite{buekenhout2013diagram,ARP,Bourbaki2006}.

In general, the automorphism group $\Aut(\beta)$ of a regular hypertope $\beta$ is not a Coxeter group, but only a well-behaved quotient of a Coxeter group $W$. In particular, the group $\Aut(\beta)$ is always generated by involutions $\{\rho_0,\rho_1,\ldots,\rho_{n-1}\}$. We can then define the \textit{Coxeter diagram} of a regular hypertope $\beta$ to be the Coxeter graph $\mathcal{D}$ if the covering Coxeter group of $\Aut(\beta)$ is $W(\mathcal{D})$. More precisely, the Coxeter diagram of $\beta$ is a graph whose vertices correspond to generators the $\rho_0, \rho_1, \ldots, \rho_{n-1}$ of $\Aut(\beta)$ and where the label of the edge between $\rho_i$ and $\rho_j$ is given by the order of $\rho_i\rho_j$, for any $i \neq j =0,1,\ldots,n-1$. Here also, if an edge should receive the label $2$, it is instead omitted. This is a special case of the more general notion of \textit{Buekenhout diagram} of an incidence system.

\subsection{Twisting of coset incidence systems}

Let $\alpha=(A,(A_i)_{i\in I_\alpha})$ and $\beta=(B,(B_i)_{i\in I_\beta})$ be coset incidence systems, and let
\[
\eta:B\longrightarrow\Aut(A)
\]
be an action that permutes the maximal parabolic subgroups of $\alpha$. This action induces an action of $B$ on $I_\alpha$. Let $\mathcal{K}$ be the set of orbits of this action. For each $L\in\mathcal{K}$, choose a representative $F_L\in L$. If $M\subseteq I_\beta$, let
$$
\mathcal{O}_L(M)=B_{I_\beta\setminus M}\cdot F_L.
$$
Following~\cite[Definition~3.8]{wreath}, the system $\alpha$ is \emph{$(\beta,\eta)$-admissible} when the action permutes its maximal parabolic subgroups and, for every $L\in\mathcal{K}$ and every $M,N\subseteq I_\beta$,
\begin{equation}\label{eq:IPO}
\mathcal{O}_L(M)\cap\mathcal{O}_L(N)=\mathcal{O}_L(M\cap N).
\end{equation}

\begin{definition}[Twisting of $\alpha$ by $\beta$]\cite{wreath}\label{def:gen_twist}
    Suppose that $\alpha=(A,(A_i)_{i\in I_\alpha})$ and $\beta=(B,(B_i)_{i\in I_\beta})$ are coset incidence systems and that $B$ acts on $A$ such that $\alpha$ is $(\beta,\eta)$-admissible. Then, the \textit{twisting} of $\alpha$ by $\beta$, with respect to $\eta$ and a choice of representative $\{F_L\}_{L \in K}$ for the orbits, is the coset incidence system $\TT(\alpha,\beta) = (G,(G_i)_{i \in I})$, where
    the maximal parabolic subgroups $(G_i)_{i \in I}$ are:

\begin{equation}
 G_i= \begin{cases}
     A_{(\cup_{L\in K}(L \setminus \mathcal{O}_L^i))} \rtimes B_i, & \text{if $i \in I_\beta$}.\\
     A_i \rtimes B, & \text{if $i \in K$};
  \end{cases}
\end{equation}
where $\mathcal{O}_L^i:=\mathcal{O}_L(I_\beta\setminus \{i\})=B_i\cdot F_L$.
\end{definition}

We use the following results throughout the paper.

\begin{thm}[{\cite[Corollary~3.17]{wreath}}]\label{thm:twisting-hypertope}
Let $\alpha$ and $\beta$ be regular hypertopes. If $\alpha$ is $(\beta,\eta)$-admissible, then the twisting $\TT(\alpha,\beta)$ is a regular hypertope. Moreover, $\TT(\alpha,\beta)$ has finitely many elements if and only if both $\alpha$ and $\beta$ have finitely many elements.
\end{thm}

Sometimes, we would like to prove $(\beta,\eta)$-admissibility on quotients of $A$. Let $M\leq A$. We say \emph{$M$ is closed under $\eta \colon B \to \Aut(A)$} if for all $b\in B$ we have $\eta(b)(M)= M$.

\begin{lemma}\cite[Lemma 5.8]{PiedadeTranchida_FromGroupOps2026}\label{lem:normalclosureClosedUnderGraphAut}
    Let $M\leq A$.
    If $M$ is closed under $\eta$, then the normal closure $\langle\langle M\rangle\rangle_A$ is also closed under $\eta$.
\end{lemma}

\begin{lemma}\cite[adapted from Proposition 5.10]{PiedadeTranchida_FromGroupOps2026}\label{lem:quotient-admissible}
    Let $M\leq A$ be a subgroup closed under $\eta$ and $N = \langle\langle M\rangle\rangle_A$.
    If $\alpha$ is $(\beta,\eta)$-admissible, then $\tilde{\alpha}=(\tilde{A},(\tilde{A}_i)_{i\in I_\alpha})$ is $(\beta,\eta)$-admissible, where $\tilde{A}=A/N$ and $\tilde{A_i} = A_iN/N$.
\end{lemma}

\section{The $j$-diagonals Twisting}\label{sec:twisting}

Let $\beta=(B,(B_i)_{i\in I_\beta})$ be a coset incidence system and fix $j\in I_\beta$. We write $X_j=\{\sigma B_j\mid \sigma \in B\}$
for the set of $j$-elements of $\beta$. We will use the natural action of $B$ on this set $X_j$ to construct a graph on which $B$ acts.

\subsection{Diagonal classes and their Coxeter graphs}\label{subsec:diagonals}

We start by defining diagonal classes and prove an elementary result about them. 

\begin{definition}[Diagonals and diagonal classes]\label{def:diagonal}
A \emph{$j$-diagonal} of $\beta$ is a an unordered pair $\{F_1,F_2\}$ of distinct elements of $X_j$.

Two $j$-diagonals $\{F_1,F_2\}$ and $\{F_3,F_4\}$ are \emph{equivalent} if there is some $b\in B$ such that
\[
\{F_3,F_4\}=\{bF_1,bF_2\}.
\]
The equivalence classes are called \emph{$j$-diagonal classes}. The class containing $\{F_1,F_2\}$ is denoted by $[F_1,F_2]$, and the set of all $j$-diagonal classes is denoted by $\CC^j_\beta$.
\end{definition}

Every $j$-diagonal can be written as $\{\sigma_1B_j,\sigma_2B_j\}$ with $\sigma_1B_j\neq\sigma_2B_j$. Since $B$ is transitive on $X_j$, every diagonal class has a representative of the form $\{B_j,\sigma B_j\}$.

\begin{lemma}\label{lem:double-cosets}
Every $j$-diagonal class has a representative
$\{B_j,\tau B_j\}, \tau\notin B_j$.
Moreover, for $\sigma_1,\sigma_2\notin B_j$, the diagonals $\{B_j,\sigma_1B_j\}$ and
$\{B_j,\sigma_2B_j\}$
belong to the same diagonal class if and only if
\begin{equation}\label{eq:doublecosetdiagonal}
\sigma_2\in B_j\sigma_1B_j\cup B_j\sigma_1^{-1}B_j.
\end{equation}
\end{lemma}

\begin{proof}
Let $D=\{\sigma_1B_j,\sigma_2B_j\}$ be a $j$-diagonal. Multiplication by $\sigma_1^{-1}$ sends $D$ to
$\{B_j,\sigma_1^{-1}\sigma_2B_j\}$.
The element $\sigma_1^{-1}\sigma_2$ does not belong to $B_j$, since the two cosets in $D$ are distinct.

Now let $
D_1=\{B_j,\sigma_1B_j\}$ and
$D_2=\{B_j,\sigma_2B_j\}$. Suppose first that $\mu D_1=D_2$ for some $ \mu\in B$. Since the pairs are unordered, there are two cases. If $\mu B_j=B_j$, then $\mu \in B_j$ and $ \mu \sigma_1B_j=\sigma_2B_j$, so $\sigma_2\in B_j\sigma_1B_j$. If instead $\mu B_j=\sigma_2B_j$ and $\mu \sigma_1B_j=B_j$, then $\mu=\sigma_2\delta=\gamma\sigma_1^{-1}$ for some $\delta,\gamma\in B_j$. Hence $\sigma_2=\gamma\sigma_1^{-1}\delta^{-1}\in B_j\sigma_1^{-1}B_j$.

Conversely, suppose that~\eqref{eq:doublecosetdiagonal} holds. If $\sigma_2=\mu_1\sigma_1\mu_2$ with $\mu_1,\mu_2\in B_j$, then $\mu_1D_1=D_2$. If $\sigma_2=\mu_1\sigma_1^{-1}\mu_2$, then $\mu_1\sigma_1^{-1}D_1=D_2$, with the two entries interchanged. Thus $D_1$ and $D_2$ are equivalent.
\end{proof}

We now define a Coxeter graph from these diagonal classes.

\begin{construction}[$j$-diagonal Coxeter graph]\label{const:graphFromCosetIncidence}
Assume that $\beta=(B,(B_i)_{i\in I_\beta})$ is a coset incidence system and that its $j$-elements set $X_j$ is finite and has more than one element. Let
\[
\lambda:\CC^j_\beta\longrightarrow \{2,3,\ldots\}\cup\{\infty\}
\]
be a label function. The \emph{$j$-diagonal Coxeter graph} $\mathcal{D}=\mathcal{D}(\beta,j,\lambda)$ is defined as follows:
\begin{itemize}
    \item $V(\mathcal{D})=X_j$;
    \item two distinct vertices $F_1,F_2\in X_j$ are joined by an edge if and only if $\lambda([F_1,F_2])\neq 2$;
    \item when this edge is present, its label is $\lambda([F_1,F_2])$.
\end{itemize}
\end{construction}

The graph has no loops and is in fact a Coxeter graph, as defined in Section \ref{sec:prelims:Coxeter}. The corresponding Coxeter group is

\begin{equation*}
W(\mathcal{D})=\left\langle a_F\ \middle|\
 a_F^2=1,\ (a_Fa_E)^{\lambda([F,E])}=1
 \text{ if }F\neq E\text{ and }\lambda([F,E])<\infty
\right\rangle
\end{equation*}
where the indices $F$ and $E$ run through elements of $X_j$.
\begin{lemma}\label{lem:gammaHomoGraph}
The map $\gamma:B\longrightarrow\Aut(\mathcal{D})$ such that $\gamma(b)(xB_j)=bxB_j$
is a well-defined action of $B$ by label-preserving graph automorphisms.
\end{lemma}

\begin{proof}
If $\{F_1,F_2\}$ is a $j$-diagonal, then $\{bF_1,bF_2\}$ lies in the same diagonal class. Hence the two pairs receive the same label, so $\gamma(b)$ preserves the edges and their labels. Moreover, we have
$$
\gamma(b_1)\bigl(\gamma(b_2)(xB_j)\bigr)
=b_1b_2xB_j
=\gamma(b_1b_2)(xB_j),
$$
so $\gamma$ is a homomorphism.
\end{proof}

Note that the kernel of $\gamma$ is the core of $B_j$ in $B$. In particular, the action is faithful when $B_j$ is core-free.

Every label-preserving automorphism $\varphi$ of the graph $\mathcal{D}$ induces an automorphism of the group $A=W(\mathcal{D})$ by
$$
a_F\longmapsto a_{\varphi(F)}.
$$
Indeed, this map preserves the Coxeter presentation of $W(\mathcal{D})$ describe above. We therefore obtain a homomorphism
$$
\varphi:\Aut(\mathcal{D})\longrightarrow\Aut(A)
$$
and an action
\begin{equation}\label{eq:eta}
\eta=\varphi\circ\gamma:B\longrightarrow\Aut(A).
\end{equation}

Let
$\alpha=(A,(A_F)_{F\in X_j})$, where $A_F=\langle a_E\mid E\in X_j\setminus\{F\}\rangle$,
be the standard coset incidence system of the Coxeter group $A = W(\mathcal{D})$. We now investigate when the action $\alpha$ is $(\beta,\varphi)$-admissible. It turns out that only the orbit intersection property has to be verified, and that it always holds if $B$ acts flag-transitively on $\beta$.

\begin{thm}\label{thm:admissibility-diagonal}
The action $\eta$ permutes the maximal parabolic subgroups of $\alpha$, and its action on the type set $X_j$ is transitive. Let $F_L=B_j$ be the choice of base point for the unique orbit $F_L$. Then $\alpha$ is $(\beta,\eta)$-admissible if and only if
\begin{equation}\label{eq:IPO-diagonal}
\mathcal O(M)\cap\mathcal O(N)=\mathcal O(M\cap N)
\qquad(M,N\subseteq I_\beta),
\end{equation}
where
$\mathcal O(M)=B_{I_\beta\setminus M}\cdot F_L.$
Moreover, if $B$ is flag-transitive in $\beta$, then $\alpha$ is always $(\beta,\eta)$-admissible.
\end{thm}

\begin{proof}
For $b\in B$ and $F\in X_j$, the automorphism $\eta(b)$ sends $a_E$ to $a_{bE}$. Hence,
we have $\eta(b)(A_F)=A_{bF}$.

Thus the maximal parabolic subgroups are permuted, and the action on their index set is simply the transitive action of $B$ on $X_j$. Since the action is transitive, there is a unique orbit $L$ on $X_j$. By choosing $F_L=B_j$ as our standard base point of this orbit, the admissibility condition reduces to equation~\eqref{eq:IPO-diagonal}.

Suppose now that $B$ is flag-transitive on $\beta$. As $F_L=B_L$, we have that 
$$\mathcal{O}(M)\cap\mathcal{O}(N)= B_{I_\beta\setminus M}\cdot F_L \cap B_{I_\beta\setminus N}\cdot F_L
= (B_{I_\beta\setminus M}\cdot B_j) \cap (B_{I_\beta\setminus M}\cdot B_j),$$
for some $M,N\subseteq I_\beta$.
By the flag-transitivity of $B$ in $\beta$ (see Theorem~\ref{thm:cosetFT}), we have
$$(B_{I_\beta\setminus M}\cdot B_j) \cap (B_{I_\beta\setminus N}\cdot B_j) = (B_{(I_\beta\setminus M)\cap (I_\beta\setminus N)}\cdot B_j) = \mathcal{O}(M\cap N),
$$
concluding the proof.
\end{proof}

\subsection{The $j$-diagonals twisting}

Starting from a regular hypertope $\beta =(B,(B_j)_{j \in I_\beta})$ and choosing a type $j \in I$, we can construct a $j$-diagonal Coxeter graph $\mathcal{D}$ on which $B$ acts. We can now apply the twisting construction of \cite{wreath} to this setting.

\begin{thm}\label{thm:diagonal-twisting}
Let $\beta=(B,(B_i)_{i\in I_\beta})$ be a regular hypertope. Fix $j\in I_\beta$, assume that $X_j$ is finite and has at least two elements, and let $\mathcal{D}$ be a $j$-diagonal Coxeter graph. Consider $A=W(\mathcal{D})$ and let $\alpha$ be its standard coset incidence system. Then:
\begin{enumerate}
    \item $\TT(\alpha,\beta)$ is a regular hypertope, for $F_L=B_j$;
    \item If $M\leq A$ is invariant under $\eta(B)$, let $N=\langle\!\langle M\rangle\!\rangle_A$ and set
    \[
    \widetilde\alpha=
    \bigl(A/N,(A_FN/N)_{F\in X_j}\bigr).
    \]
    If $\widetilde\alpha$ is a regular hypertope, then $\TT(\widetilde\alpha,\beta)$ is a regular hypertope for $F_L=B_j$.
\end{enumerate}
\end{thm}

\begin{proof}
The standard coset incidence system of a Coxeter group is a regular hypertope. By Theorem~\ref{thm:admissibility-diagonal} we know that $\alpha$ is $(\beta,\eta)$-admissible since $\beta$ is a regular hypertope. Part~(a) thus follows from Theorem~\ref{thm:twisting-hypertope}.

For part~(b), the invariance of $M$ implies that $N$ is invariant under $\eta(B)$. Lemma~\ref{lem:quotient-admissible} shows that $\widetilde\alpha$ is admissible. The result again follows from Theorem~\ref{thm:twisting-hypertope}.
\end{proof}

The result above mimics a construction of a generalised cube from a polytope $\beta$, denoted as $2^\beta$, that was introduced by Danzer~\cite{danzer_regular_1984} and revised in~\cite[Theorem 8C2]{ARP}. Indeed, Danzer's construction from a regular polytope will coincide with our construction when the edge labels of $\mathcal{D}$ are all equal to $2$.
Additionally, in \cite{ARP}, McMullen and Schulte define a polytope construction $2^{\beta,\mathcal{D}(s)}$ for centrally symmetric regular polytopes $\beta$, where $\mathcal{D}(s)$ is a Coxeter graph with labels $2$ for all $0$-diagonals, except those connecting antipodal vertices of $\beta$, which are connected by an edge labelled $s$. Indeed, $2^\beta = 2^{\beta,\mathcal{D}(2)}$. In general, one can define the polytope $2^{\beta,\mathcal{D}}$, for any diagram $\mathcal{D},$ as given in~\cite[Section~8C]{ARP}, which will match construction given in Theorem~\ref{thm:diagonal-twisting} if we fix $F_L=B_0$.

\begin{coro}\label{coro:twistingGiving2PG}
Let $\beta$ be an abstract regular $n$-polytope with automorphism group
$B=\langle\rho_0,\ldots,\rho_{n-1}\rangle$.
Let $\mathcal{D}$ be the Coxeter graph on the $0$-elements of $\beta$ obtained from Construction~\ref{const:graphFromCosetIncidence}.
Choose the base vertex $F_L=B_0$. Then $2^{\beta,\mathcal{D}}\cong\TT(\alpha,\beta)$,
where $\alpha$ is the standard coset incidence system of $W(\mathcal{D})$.
\end{coro}
\begin{proof}
The action of $B$ on $\mathcal{D}$ is the action used in the classical construction. The choice $F_L=B_0$ ensures that the stabilizer $B_0$ fixes the base vertex. The two constructions therefore give the same distinguished generators and the same maximal parabolic subgroups. Hence the associated regular polytopes are isomorphic.
\end{proof}

We now describe the diagram of the hypertope $\TT(\alpha,\beta)$ obtained in Theorem~\ref{thm:diagonal-twisting}. It is an extension of the diagram of $\beta$ by an even edge, attached to the $j$-node of the diagram of $\beta$. 

\begin{coro}\label{coro:diagramtwistingGraph}Suppose we are in the settings of the hypothesis of Theorem~\ref{thm:diagonal-twisting}.
Let $I_\beta=\{1,\ldots,n\}$ and set $B=\langle b_i\mid i\in I_\beta\rangle$. Assume also that the Coxeter graph of $B$ is the following:
   
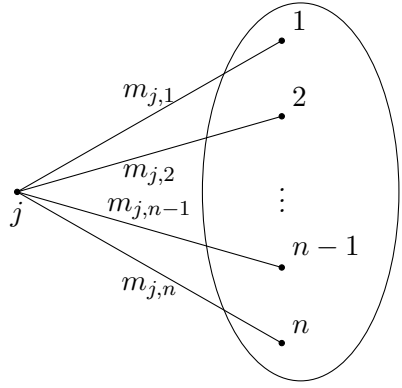
\begin{center}
    \begin{tikzpicture}[scale = 0.5]

   \filldraw[black] (2,0) circle (2pt)  node[anchor=north]{$j$};
   \filldraw[black] (7+2,0.8) circle (0pt)  node[anchor=north]{$\vdots$};
    \filldraw[black] (7+2,2) circle (2pt)  node[anchor=south west]{$2$};
         \filldraw[black] (7+2,-2) circle (2pt)  node[anchor=south west]{${n-1}$};
         \filldraw[black] (7+2,4) circle (2pt)  node[anchor=south west]{$1$};
         \filldraw[black] (7+2,-4) circle (2pt)  node[anchor=south west]{${n}$};
    \draw (2,0) -- (7+2,2)node [midway, below] (TextNode) {$m_{j,2}$};
    \draw (2,0) -- (7+2,4)node [midway,above] (TextNode) {$m_{j,1}$};
    \draw (2,0) -- (7+2,-2)node[midway,above] (TextNode) {$m_{j,n-1}$};
    \draw (2,0) -- (7+2,-4)node[midway,below] (TextNode) {$m_{j,n}$};
    \draw (7.3+2.2,0) ellipse (2.6cm and 5cm);

    \end{tikzpicture}
    \end{center}
Let $A=W(\mathcal{D})$ be the Coxeter group associated to the $j$-diagonal graph $\mathcal{D}$, generated by $\{ a_F\mid F\in V(\mathcal{D})\}$, and set $F_L = B_j$. Then the Coxeter diagram of the hypertope $\TT(\alpha,\beta)$ is
   
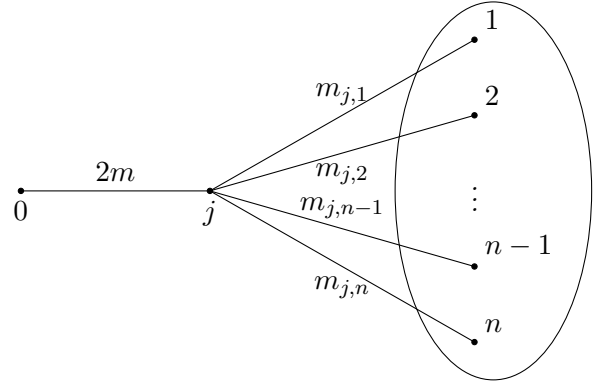
\begin{center}
    \begin{tikzpicture}[scale = 0.5]
    
   \filldraw[black] (-3,0) circle (2pt)  node[anchor=north]{$ 0 $};
   \filldraw[black] (2,0) circle (2pt)  node[anchor=north]{$j$};
   \filldraw[black] (7+2,0.8) circle (0pt)  node[anchor=north]{$\vdots$};
    \filldraw[black] (7+2,2) circle (2pt)  node[anchor=south west]{$2$};
         \filldraw[black] (7+2,-2) circle (2pt)  node[anchor=south west]{${n-1}$};
         \filldraw[black] (7+2,4) circle (2pt)  node[anchor=south west]{$1$};
         \filldraw[black] (7+2,-4) circle (2pt)  node[anchor=south west]{${n}$};
    \draw (2,0) -- (-3,0)node [midway,above] (TextNode) {$2m$};
    \draw (2,0) -- (7+2,2)node [midway, below] (TextNode) {$m_{j,2}$};
    \draw (2,0) -- (7+2,4)node [midway,above] (TextNode) {$m_{j,1}$};
    \draw (2,0) -- (7+2,-2)node[midway,above] (TextNode) {$m_{j,n-1}$};
    \draw (2,0) -- (7+2,-4)node[midway,below] (TextNode) {$m_{j,n}$};
    \draw (7.3+2.2,0) ellipse (2.6cm and 5cm);

    \end{tikzpicture}
    \end{center}
where $m = \lambda([B_j, b_jB_j])$.
\end{coro}
\begin{proof}
    First, note that the automorphism group of the regular hypertope $\TT(\alpha,\beta)$ is $A\rtimes_\eta B$, that is 
    $$\langle a_F \mid F\in V(\mathcal{D}) \rangle \rtimes_\eta \langle b_i\mid i\in \{1,\ldots,n\} \rangle.$$
    Indeed, given the group presentations for $A = \langle a_k, k\in V(\mathcal{D})\mid R_A\rangle$ and for $B = \langle b_i, i\in I_\beta\mid R_B\rangle$, we have $$A\rtimes_\eta B = \langle a_k,b_i (k\in V(\mathcal{D}),i\in I_\beta)\mid R_A, R_B, \eta(b)(a) = bab^{-1}, \forall a\in A\ \forall b\in B\rangle$$
    As the action of $B$ on $V(\mathcal{D})$ is transitive, we can write all generators of $A$ in terms of the generator $a_{B_j}$.  This means that for every $F\in V(\mathcal{D})$, there exists $b\in B$ such that $\eta(b)(a_{B_j})=a_F$, which can be rewritten as $ba_{B_j}b^{-1}=a_F$, since $\eta$ acts by conjugation, as shown by the group presentation.
    Therefore, the group $A\rtimes_\eta B = \langle a_{B_j}, b_i\mid i\in I_\beta\rangle$, since the other generators of $A$ can be obtained by letting $\eta$ act on $a_{B_j}$.
    We now need to determine the Coxeter relations between the generators $\{a_{B_j},b_i\mid i\in I_\beta\}$. The relations between $b_i$ and $b_j$, for $i,j \in I_\beta$, are the same as in $B$. We only need to check the relation between $a_{B_j}$ and $b_i$, for $i\in I_\beta$.
    For every $i\neq j$, note that $\eta(b_i)(a_{B_j}) = a_{b_iB_j} = a_{B_j}$, since $b_i\in B_j$. Hence, since $a_{B_j} = \eta(b_i)(a_{B_j}) = b_ia_{B_j}b_i$, we conclude that $(b_ia_{B_j})^2 = e$. In the Coxeter graph of $A\rtimes_\eta B$, this implies that the vertices $a_{B_j}$ and $b_i$ are disconnected for $i\neq j$.
    For $i=j$, we have that $\eta(b_j)(a_{B_j}) = a_{b_jB_j}=b_ja_{B_j}b_j$. From the $j$-diagonal graph $\mathcal{D}$, we know that $m=\lambda([B_j,b_jB_j])$. Therefore, we have $(a_{B_j}a_{b_jB_j})^m=e$. By substituting $a_{b_jB_j}$ with $b_ja_{B_j}b_j$, we get $e = (a_{B_j}b_ja_{B_j}b_j)^m = (a_{B_j}b_j)^{2m}$. This means that the edge connecting the vertex $a_{B_{j}}$ and the vertex $b_j$ is labelled $2m$.
    Relabelling the generators $g_0 := a_{B_j}$ and $g_i:=b_i$, for $i\in I_\beta$, we get the claimed Coxeter diagram for $A\rtimes_\eta B$.

\end{proof}

From Corollary~\ref{coro:diagramtwistingGraph}, it is very easy to see that we can always extend a regular hypertope with an even edge. Indeed, if all labels of the graph $\mathcal{D}$ are $2$, then the construction above will give a finite regular hypertope with Coxeter diagram having a new edge labelled $4$, provided the starting hypertope $\beta$ was finite. In general, we can prove the theorem below.

\begin{thm}
     Let $\mathcal{D}$ be a Coxeter graph, and suppose that it is a tree all of whose labels are equal to $4$ except possibly one label which can have any integer $k\geq 3$. Then, there exists a finite hypertope with Coxeter diagram $\mathcal{D}$.
\end{thm}
\begin{proof} 
    We will prove this inductively on the rank $n$ of the hypertope, or equivalently on the number of vertices of the tree.
    The statement holds for $n=2$. Indeed, for $k\in\{3,\ldots\}$, a $k$-polygon is a hypertope for a tree with two nodes and label $k$. It will be finite precisely when $k < \infty$. 
    
    Suppose now that the statement holds for $n-1$. We can thus find a finite $(n-1)$-hypertope $\beta$ whose Coxeter graph a tree with $n-1$ vertices whose labels satisfy the conditions of the theorem.
    Choose a type $j$ and build the $j$-diagonal Coxeter graph $\mathcal{D}$ where $\lambda(D) = 2 $ for any diagonal $D\in \CC^{j}_\beta$. In other words, this is the Coxeter graph with $[B:B_j]$ vertices and no edges. The Coxeter group $A=W(\mathcal{D})$ is the finite abelian Coxeter group of order $2^{[B:B_j]}$. By Theorem~\ref{thm:diagonal-twisting} and Corollary~\ref{coro:diagramtwistingGraph}, we get that $\TT(\alpha,\beta)$ is a finite hypertope, of order $2^{[B:B_j]}\cdot |B|$, whose Coxeter diagram is obtained from the one of $B$, by adding new vertex connected only to the vertex $j$ by edge labelled $4$. We thus have constructed a suitable hypertope of rank $n$, concluding the induction step of the proof.

\end{proof}

\subsection{Extending a Coxeter graph}\label{subsec:extension}

We now combine a given Coxeter graph with a $j$-diagonal Coxeter graph. For a coset geometry $\beta=(B,(B_i)_{i\in I_\beta})$, this will allow the natural action of $B$ on its $j$-elements to extend to an action of $B$ on a bigger Coxeter graph.
This is the graph operation needed for the general twisting construction in~\cite[Section~8B]{ARP}, that we generalize from polytopes to the context of coset geometries.

\begin{construction}[Extension of Coxeter graphs]\label{const:graphJoiningTwoGraphs}
Let $\GG$ and $\mathcal{D}$ be Coxeter graphs with distinguished vertices $v_\GG\in V(\GG)$ and $v_\mathcal{D}\in V(\mathcal{D})$. Assume that their vertex sets are otherwise disjoint. Identify $v_\GG$ with $v_\mathcal{D}$, and denote the identified vertex by $v$. The graph
\[
\KK=\EXT(\GG,\mathcal{D})
\]
is defined as follows:
\begin{itemize}
    \item $V(\KK)$ is obtained from $V(\GG)\sqcup V(\mathcal{D})$ by identifying $v_\GG$ and $v_\mathcal{D}$;
    \item all edges of $\GG$ and $\mathcal{D}$ remain, with the same labels;
    \item if $w\in V(\GG)\setminus\{v\}$ is joined to $v$ in $\GG$ with label $m$, then $w$ is joined to every $w'\in V(\mathcal{D})\setminus\{v\}$ with the same label $m$.
\end{itemize}
No other edges are added.
\end{construction}

One important thing to notice is that $\EXT(\GG,\mathcal{D})$ is not the same as $\EXT(\mathcal{D},\GG)$.

\begin{example}\label{ex:extension}
Let the vertices of $\GG$ be $u,v_\GG,w$, from left to right, and let the vertices of $\mathcal{D}$ be $h_1,v_\mathcal{D},h_2$, from left to right:
$$
\GG:\quad
\xymatrix@C=18mm{
*{\bullet}\ar@{-}[r]^3 &*{\bullet}\ar@{-}[r]^4 & *{\square}\ar@{-}[r]^\infty & *{\bullet}}
\qquad
\mathcal{D}:\quad
\xymatrix@C=18mm{
*{\bullet}\ar@{-}[r]^5 & *{\square}\ar@{-}[r]^3 & *{\bullet}}
$$
After identifying the two vertices $v_{\GG}$ and $v_{\mathcal{D}}$, marked as squares in the diagrams, the extension $\EXT(\GG,\mathcal{D})$ is
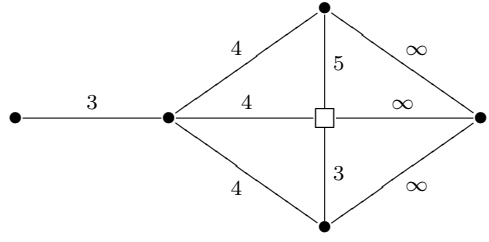
$$
\xymatrix@C=18mm@R=12mm{
& & *{\bullet}\ar@{-}[d]^5\ar@{-}[dl]_4\ar@{-}[dr]^\infty & \\
*{\bullet}\ar@{-}[r]^3 &*{\bullet}\ar@{-}[r]^4 & *{\square}\ar@{-}[r]^\infty\ar@{-}[d]^3 & *{\bullet}\\
& & *{\bullet}\ar@{-}[ul]^4\ar@{-}[ur]_\infty &
}$$
Thus the neighbours of $v_\GG$ in $\GG$ are connected to every vertex of $\mathcal{D}$ with the same labels as before.
On the other hand, the extension $\EXT(\mathcal{D},\GG)$ gives another Coxeter graph.
$$
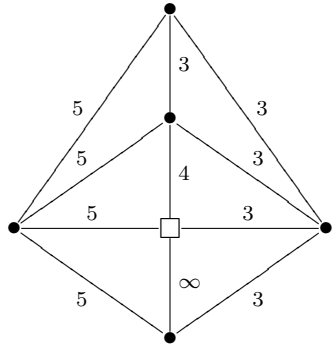
\xymatrix@C=18mm@R=12mm{
& *{\bullet}\ar@{-}[d]^3\ar@{-}[ddl]_5\ar@{-}[ddr]^3 \\
& *{\bullet}\ar@{-}[d]^4\ar@{-}[dl]_5\ar@{-}[dr]^3 & \\
*{\bullet}\ar@{-}[r]^5 & *{\square}\ar@{-}[r]^3\ar@{-}[d]^\infty & *{\bullet}\\
& *{\bullet}\ar@{-}[ul]^5\ar@{-}[ur]_3 &
}$$

\end{example}

If $\GG$ has only one vertex, then $\EXT(\GG,\mathcal{D})=\mathcal{D}$. Similarly, if $\mathcal{D}$ has one vertex, then $\EXT(\GG,\mathcal{D})=\GG$.

Now let $\mathcal{D}$ be a $j$-diagonal Coxeter graph of $\beta$, and choose $v_\mathcal{D}=B_j$. Let $\KK=\EXT(\GG,\mathcal{D})$. The action of $B$ on $\mathcal{D}$ extends to $\KK$ by fixing each vertex of $V(\GG)\setminus\{v\}$.

\begin{lemma}\label{lem:gammaHomoExtendedGraph}
The map
\[
\gamma:B\longrightarrow\Aut(\KK)
\]
defined by
\[
\gamma(b)(g)=g\quad(g\in V(\GG)\setminus\{v\}),
\qquad
\gamma(b)(h)=bh\quad(h\in V(\mathcal{D}))
\]
is a well-defined action by label-preserving graph automorphisms.
\end{lemma}

\begin{proof}
Edges contained in $\mathcal{D}$ are preserved by Lemma~\ref{lem:gammaHomoGraph}. Vertices in $V(\GG)\setminus\{v\}$ are fixed. If $g$ is joined to one vertex of $\mathcal{D}$ by an edge copied from the edge $\{g,v\}$ of $\GG$, then Construction~\ref{const:graphJoiningTwoGraphs} joins $g$ to every vertex of $\mathcal{D}$ with the same label. Hence the image of this edge is again an edge with the same label. The homomorphism property follows from the action of $B$ on $V(\mathcal{D})$.
\end{proof}

Let $A=W(\KK)$ and let $\alpha=(A,(A_i)_{i\in V(\KK)})$ be its standard coset incidence system. As before, the action on the graph induces an action
\[
\eta:B\longrightarrow\Aut(A).
\]
The action on $V(\KK)$ has one orbit $V(\mathcal{D})$ and fixes every vertex in $V(\GG)\setminus\{v\}$. The orbit intersection condition is automatic for the singleton orbits. Thus only the orbit $V(\mathcal{D})$ must be checked, and this will follow again directly from the fact that we will take $\beta$ to be a regular hypertope (see Theorem~\ref{thm:admissibility-diagonal}).

\begin{thm}\label{thm:extended-twisting}
Let $\beta$ be a regular hypertope, let $\mathcal{D}$ be a $j$-diagonal Coxeter graph, and let $\GG$ be a finite Coxeter graph with a distinguished vertex. Form $\KK=\EXT(\GG,\mathcal{D})$ by identifying the distinguished vertex of $\GG$ with $B_j\in V(\mathcal{D})$. Let $A=W(\KK)$ and let $\alpha$ be its standard coset incidence system. Then:
\begin{enumerate}
    \item $\TT(\alpha,\beta)$ is a regular hypertope, for $F_L=B_j$;
    \item If $M\leq A$ is invariant under $\eta(B)$, let $N=\langle\!\langle M\rangle\!\rangle_A$ and set
    \[
    \widetilde\alpha=
    \bigl(A/N,(A_iN/N)_{i\in V(\KK)}\bigr).
    \]
    If $\widetilde\alpha$ is a regular hypertope, then $\TT(\widetilde\alpha,\beta)$ is a regular hypertope, for $F_L=B_j$.
\end{enumerate}
\end{thm}

\begin{proof}
The graph $\KK$ defines a Coxeter group, so $\alpha$ is a regular hypertope. Lemma~\ref{lem:gammaHomoExtendedGraph} gives the action of $B$, and the discussion before the theorem shows that we will have our admissibility condition. Part~(a) follows from Theorem~\ref{thm:twisting-hypertope}. Part~(b) follows from Lemma~\ref{lem:quotient-admissible} and the same theorem.
\end{proof}

When $\GG$ consists of one vertex, Theorem~\ref{thm:extended-twisting} reduces to Theorem~\ref{thm:diagonal-twisting}.

Again, we now make the connection with the classical results of \cite{ARP}, We refer to~\cite[Section~8B]{ARP} for the details of the construction of the regular polytopes $\delta^{\beta,\GG}$.

\begin{coro}\label{coro:classical-general-twisting}
Let $\delta=(C,(C_i)_{i\in I_\delta})$ be a universal regular $m$-polytope, where
\[
C=\langle c_0,c_1,\ldots,c_{m-1}\rangle
\]
is a string Coxeter group with Coxeter graph $\GG$. Let $\beta=(B,(B_i)_{i\in I_\beta})$ be an abstract regular polytope, and let $\mathcal{D}$ be the Coxeter graph obtained from the $0$-diagonals of $\beta$. Identify the vertex of $\GG$ corresponding to $c_{m-1}$ with the vertex $B_0$ of $\mathcal{D}$, and let $\KK=\EXT(\GG,\mathcal{D})$. If $F_L = B_0$ then,
\[
\delta^{\beta,\GG}\cong\TT(\alpha,\beta),
\]
where $\alpha$ is the standard coset incidence system of $W(\KK)$.
\end{coro}

\begin{proof}
The Coxeter group $W(\KK)$ and its standard parabolic subgroups are exactly the group and distinguished parabolic subgroups used in the classical construction. The choice of the vertex $B_0$ gives the required fixed base vertex. Thus the two constructions define the same regular polytope.
\end{proof}

\begin{coro}\label{coro:extensiontwistingGraph}\
Suppose we are in the settings of the hypothesis of Theorem~\ref{thm:extended-twisting}.
Let $I_\beta=\{1,\ldots,n\}$ and set $B=\langle b_i\mid i\in I_\beta\rangle$. Assume also that the Coxeter graph of $B$ is the following:

   
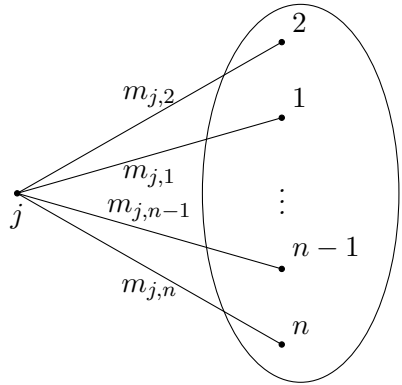
\begin{center}
    \begin{tikzpicture}[scale = 0.5]

   \filldraw[black] (2,0) circle (2pt)  node[anchor=north]{$j$};
   \filldraw[black] (7+2,0.8) circle (0pt)  node[anchor=north]{$\vdots$};
    \filldraw[black] (7+2,2) circle (2pt)  node[anchor=south west]{$1$};
         \filldraw[black] (7+2,-2) circle (2pt)  node[anchor=south west]{${n-1}$};
         \filldraw[black] (7+2,4) circle (2pt)  node[anchor=south west]{$2$};
         \filldraw[black] (7+2,-4) circle (2pt)  node[anchor=south west]{${n}$};
    \draw (2,0) -- (7+2,2)node [midway, below] (TextNode) {$m_{j,1}$};
    \draw (2,0) -- (7+2,4)node [midway,above] (TextNode) {$m_{j,2}$};
    \draw (2,0) -- (7+2,-2)node[midway,above] (TextNode) {$m_{j,n-1}$};
    \draw (2,0) -- (7+2,-4)node[midway,below] (TextNode) {$m_{j,n}$};
    \draw (7.3+2.2,0) ellipse (2.6cm and 5cm);

    \end{tikzpicture}
    \end{center}

Let $\GG$ be the finite Coxeter graph

   
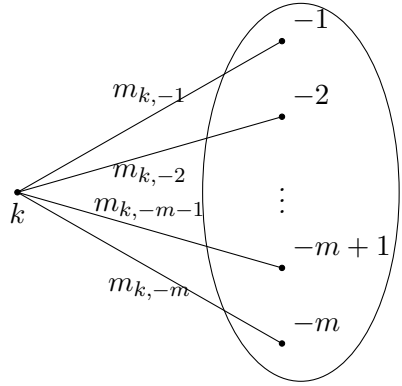
\begin{center}
    \begin{tikzpicture}[scale = 0.5]
   \filldraw[black] (2,0) circle (2pt)  node[anchor=north]{$k$};
   \filldraw[black] (7+2,0.8) circle (0pt)  node[anchor=north]{$\vdots$};
    \filldraw[black] (7+2,2) circle (2pt)  node[anchor=south west]{${-2}$};
         \filldraw[black] (7+2,-2) circle (2pt)  node[anchor=south west]{${-m+1}$};
         \filldraw[black] (7+2,4) circle (2pt)  node[anchor=south west]{${-1}$};
         \filldraw[black] (7+2,-4) circle (2pt)  node[anchor=south west]{${-m}$};
    \draw (2,0) -- (7+2,2)node [midway, below] (TextNode) {$m_{k,-2}$};
    \draw (2,0) -- (7+2,4)node [midway,above] (TextNode) {$m_{k,-1}$};
    \draw (2,0) -- (7+2,-2)node[midway,above] (TextNode) {$m_{k,-m-1}$};
    \draw (2,0) -- (7+2,-4)node[midway,below] (TextNode) {$m_{k,-m}$};
    \draw (7.3+2.2,0) ellipse (2.6cm and 5cm);

    \end{tikzpicture}
    \end{center}
with vertex set $I_\GG=\{-m,\ldots,-1\}$, with selected vertex $k\in I_\GG$ of $\GG$ such that it will be identified with the vertex $B_j$ of $\mathcal{D}$.
Let $A=W(\EXT(\GG,\mathcal{D}))$ be the extended Coxeter group, and set $F_L = B_j$. Then the Coxeter diagram of the hypertope $\TT(\alpha,\beta)$ is
   
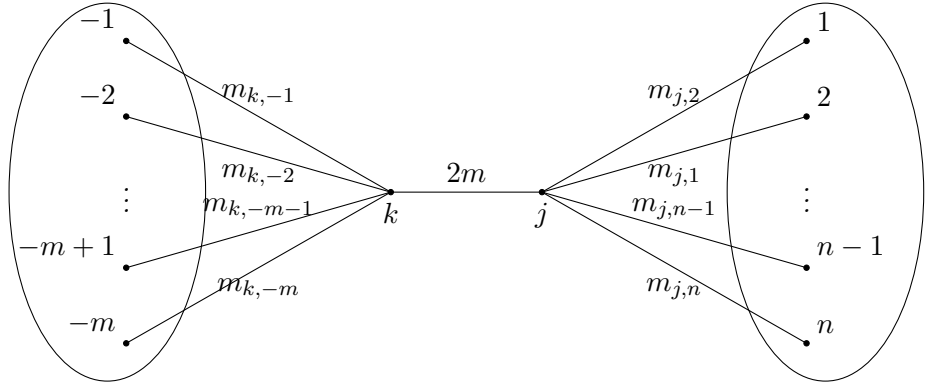
\begin{center}
    \begin{tikzpicture}[scale = 0.5]
    
   \filldraw[black] (-2,0) circle (2pt)  node[anchor=north]{$ k $};
   \filldraw[black] (2,0) circle (2pt)  node[anchor=north]{$j$};
   \filldraw[black] (7+2,0.8) circle (0pt)  node[anchor=north]{$\vdots$};
    \filldraw[black] (7+2,2) circle (2pt)  node[anchor=south west]{$2$};
         \filldraw[black] (7+2,-2) circle (2pt)  node[anchor=south west]{${n-1}$};
         \filldraw[black] (7+2,4) circle (2pt)  node[anchor=south west]{$1$};
         \filldraw[black] (7+2,-4) circle (2pt)  node[anchor=south west]{${n}$};
    
    \filldraw[black] (-7-2,0.8) circle (0pt)  node[anchor=north]{$\vdots$};
    \filldraw[black] (-7-2,2) circle (2pt)  node[anchor=south east]{$-2$};
         \filldraw[black] (-7-2,-2) circle (2pt)  node[anchor=south east]{${-m+1}$};
         \filldraw[black] (-7-2,4) circle (2pt)  node[anchor=south east]{$-1$};
         \filldraw[black] (-7-2,-4) circle (2pt)  node[anchor=south east]{$-{m}$};
    \draw (2,0) -- (-2,0)node [midway,above] (TextNode) {$2m$};
     \draw (2,0) -- (7+2,2)node [midway, below] (TextNode) {$m_{j,1}$};
    \draw (2,0) -- (7+2,4)node [midway,above] (TextNode) {$m_{j,2}$};
    \draw (2,0) -- (7+2,-2)node[midway,above] (TextNode) {$m_{j,n-1}$};
    \draw (2,0) -- (7+2,-4)node[midway,below] (TextNode) {$m_{j,n}$};
    \draw (7.3+2.2,0) ellipse (2.6cm and 5cm);

     \draw (-2,0) -- (-7-2,2)node [midway, below] (TextNode) {$m_{k,-2}$};
    \draw (-2,0) -- (-7-2,4)node [midway,above] (TextNode) {$m_{k,-1}$};
    \draw (-2,0) -- (-7-2,-2)node[midway,above] (TextNode) {$m_{k,-m-1}$};
    \draw (-2,0) -- (-7-2,-4)node[midway,below] (TextNode) {$m_{k,-m}$};
    \draw (-7.3-2.2,0) ellipse (2.6cm and 5cm);

    \end{tikzpicture}
    \end{center}
where $m = \lambda([B_j,\rho_jB_j])$.
\end{coro}
\begin{proof}
    This proof follows the exact same arguments as the proof of Corollary~\ref{coro:diagramtwistingGraph}.
\end{proof}

\bibliographystyle{ieeetr}
\bibliography{refs}

\end{document}